\documentclass[11pt]{amsart}
\usepackage{amssymb}
\usepackage[all]{xy}
\usepackage{graphicx} 
\usepackage{cleveref}

\theoremstyle{plain}
\newtheorem{theorem}{Theorem}[section]
\newtheorem{proposition}[theorem]{Proposition}
\newtheorem{lemma}[theorem]{Lemma}
\newtheorem{corollary}[theorem]{Corollary}
\crefname{theorem}{Theorem}{Theorems}
\crefname{proposition}{Proposition}{Propositions}
\crefname{lemma}{Lemma}{Lemmas}
\crefname{corollary}{Corollary}{Corollaries}

\theoremstyle{definition}

\newtheorem{example}[theorem]{Example}
\newtheorem{problem}[theorem]{Problem}
\crefname{definition}{Definition}{Definitions}
\crefname{example}{Example}{Examples}

\theoremstyle{remark}
\newtheorem{remark}[theorem]{Remark}
\crefname{remark}{Remark}{Remarks}

\numberwithin{equation}{section}
\numberwithin{figure}{section}
\numberwithin{table}{section}
\crefname{equation}{Equation}{Equations}
\crefname{figure}{Figure}{Figures}
\crefname{table}{Table}{Tables}
\crefname{section}{Section}{Sections}
\crefname{subsection}{Section}{Sections}
\crefname{appendix}{Appendix}{Appendixes}

\newcommand{\Z}{\mathbb{Z}}
\newcommand{\Q}{\mathbb{Q}}
\newcommand{\R}{\mathbb{R}}
\newcommand{\C}{\mathbb{C}}

\newcommand{\id}{\operatorname{id}}

\newcommand{\tr}{\operatorname{tr}}

\newcommand{\Hom}{\operatorname{Hom}}

\newcommand{\codim}{\operatorname{codim}}
\newcommand{\Conv}{\operatorname{Conv}}

\newcommand{\irr}{\mathrm{irr}}

\newcommand{\ang}[1]{\left\langle{#1}\right\rangle}
\newcommand{\SL}{\mathit{SL}}

\newcommand{\RT}{\mathit{RT}}
\renewcommand{\SS}{\mathit{SS}}
\newcommand{\BS}{\mathit{BS}}

\begin{document}

\title[Finiteness of the Reidemeister torsion]{Finiteness of the image of the Reidemeister torsion of a splice}
\author[T. Kitano \and Y. Nozaki]{Teruaki Kitano \and Yuta Nozaki}
\subjclass[2010]{Primary 57M27, 57M25, Secondary 20C99, 14M99}
\keywords{Reidemeister torsion, $A$-polynomial, character variety, splice, bending, Riley polynomial.}
\address{Department of Information Systems Science, Faculty of Science and Engineering, Soka University \\
Tangi-cho 1-236, Hachioji, Tokyo 192-8577 \\
Japan}
\email{kitano@soka.ac.jp}
\address{Organization for the Strategic Coordination of Research and Intellectual Properties, Meiji University \\
4-21-1 Nakano, Nakano-ku, Tokyo, 164-8525 \\
Japan
\newline
Current address:
Graduate School of Advanced Science and Engineering, Hiroshima University \\
1-3-1 Kagamiyama, Higashi-Hiroshima City, Hiroshima, 739-8526 \\
Japan}
\email{nozakiy@hiroshima-u.ac.jp}

\maketitle

\begin{abstract}
 The set $\RT(M)$ of values of the $\SL(2,\C)$-Reidemeister torsion of a 3-manifold $M$ can be both finite and infinite.
 We prove that $\RT(M)$ is a finite set if $M$ is the splice of two certain knots in the 3-sphere.
 The proof is based on an observation on the character varieties and $A$-polynomials of knots.
\end{abstract}

\setcounter{tocdepth}{1}
\tableofcontents

\section{Introduction}
\label{sec:Intro}

Let $K$ be the figure-eight knot and $E(K)$ the exterior of an open tubular neighborhood of $K$ in the 3-sphere $S^3$.
The first author \cite{Kit94b} computed the $\SL(2,\C)$-Reidemeister torsion $\tau_\rho(E(K))$ 
for any acyclic irreducible representation 
$\rho\colon \pi_1(E(K)) \to \SL(2,\C)$. 
As a consequence, for the double $M = E(K)\cup_{\id}E(K)$ of $E(K)$, the set $\RT(M)$ of values of the $\SL(2,\C)$-Reidemeister torsion $\tau_\rho(M)$ is the set of all complex numbers $\C$.
In contrast, his computation also shows that $\RT(\Sigma(K,K))$ is a finite set.
Here, for knots $K_1$ and $K_2$ in $S^3$, let $\Sigma(K_1,K_2)$ denote the closed 3-manifold $E(K_1) \cup_{h} E(K_2)$, where $h$ is an orientation-reversing homeomorphism $\partial E(K_1) \to \partial E(K_2)$ interchanging meridians and preferred longitudes of the knots.
We call $\Sigma(K_1,K_2)$ the \emph{splice} of $E(K_1)$ and $E(K_2)$ (or simply the \emph{splice} of $K_1$ and $K_2$).
By definition, a splice is an integral homology 3-sphere.
Recently, Zentner~\cite{Zen18} showed that the fundamental group of any integral homology 3-sphere $M$ admits an irreducible $\SL(2,\C)$-representation, and therefore, it is worth studying $\RT(M)$.

The purpose of this paper is to generalize the above result on splices to a certain class of knots.
We focus on the character variety $X(E(K))$ and $A$-polynomial $A_K(L,M) \in \Z[L,M]$ of a knot $K$ and prove the following main theorem and its corollary.

\begin{theorem}
\label{thm:FiniteImage}
 Suppose that knots $K_1$ and $K_2$ in $S^3$ satisfy the following conditions:
\begin{itemize}
 \item for any irreducible component $C\subset X(E(K_i))$ $(i=1,2)$, either $\dim C = 0$, or $\dim C = 1$ and its image under the map $X(E(K_i)) \to X(\partial E(K_i))$ is not a point.
 \item $\gcd(A_{K_1}(L,M), A_{K_2}(M,L)) = 1$.
\end{itemize}
 Then $\RT(\Sigma(K_1,K_2))$ is a finite set.
\end{theorem}

\begin{corollary}
\label{cor:FiniteImage_TwoBridge}
 For any $2$-bridge knots $K_1$ and $K_2$, the set $\RT(\Sigma(K_1,K_2))$ is finite.
\end{corollary}

Curtis~\cite{Cur1,Cur2} defined an $\SL(2,\C)$-Casson invariant $\lambda_{\SL(2,\C)}(M)$ for any homology 3-sphere $M$. 
Roughly speaking, this invariant counts the number of isolated points of $X(M)$. 
It is known that $\lambda_{\SL(2,\C)}(\Sigma(K_1,K_2))$ is vanishing for any $K_1,K_2$ by Boden and Curtis \cite{BoCu08}. 
By definition, this implies that there are no isolated points in $X(\Sigma(K_1,K_2))$ 
and 
any connected component of $X(\Sigma(K_1,K_2))$ has a positive dimension. 
However by the main theorem $\RT(\Sigma(K_1,K_2))$ is a finite set for any knots with the above conditions. 
In fact, we concretely describe $X(\Sigma(K,K))$ for the cases where $K$ is the trefoil knot or figure-eight knot in Section\nobreakspace \ref {sec:Computation}.

Recently Abouzaid and Manolescu defined an $\SL(2,\C)$-Floer homology and also a full Casson invariant by taking its Euler characteristic in \cite{AbMa17}. 
That is a problem to study a relation with our Reidemeister torsion for a splice.

\subsection*{Acknowledgments}
The authors would like to thank Takahiro Kitayama and Luisa Paoluzzi for useful discussions.
They also wish to express their thanks to the referee for his or her careful reading of the manuscript and for various comments.
In particular, the geometric description of 2-parameter representations in Example~\ref{example:Montesinos} is suggested by the referee.

This study was supported in part by JSPS KAKENHI 16K05161. 
The second author started this study when he visited to QGM, Aarhus University.
He would like to thank the institution for the warm hospitality.
His visit to QGM was supported by JSPS Overseas Challenge Program for Young Researchers and he was also supported by Iwanami Fujukai Foundation.

\section{Character variety, $A$-polynomial and Reidemeister torsion}
\label{sec:Review}

\subsection{Representation variety and character variety}
\label{subsec:CharacterVariety}

Let $\Gamma$ be a finitely generated group.
We define the \emph{$\SL(2,\C)$-representation variety} $R(\Gamma )$ of $\Gamma$ to be the affine algebraic set $\Hom(\Gamma, \SL(2,\C))$ over $\C$.
Considering the GIT quotient of $R(\Gamma)$ by the action of $\SL(2,\C)$ by conjugation, one obtains the \emph{$\SL(2,\C)$-character variety} $X(\Gamma ) := R(\Gamma ) /\!\!/ \SL(2,\C)$ of $\Gamma$ (see \cite[Section~2]{Heu16} for instance).
The character variety $X(\Gamma )$ is again an affine algebraic set and not necessarily irreducible.
Let $R^\irr(\Gamma)$ denote the subset of irreducible representations and $X^\irr(\Gamma)$ the image of $R^\irr(\Gamma)$ under the projection $R(\Gamma) \twoheadrightarrow X(\Gamma)$.
It is known that the induced map $R^\irr(\Gamma)/\SL(2,\C) \to X^\irr(\Gamma)$ is bijective.

We focus on the case $\Gamma  = \pi_1(M)$ for a connected compact manifold $M$ and call $R(M) := R(\pi_1(M))$ (resp.\ $X(M) := X(\pi_1(M))$) the representation variety (resp.\ character variety) of $M$.
For instance, the character variety of a torus $T^2$ is described explicitly as follows:
Let $\lambda$, $\mu$ be generators of $\pi_1(T^2) = \Z^2$ and $\rho \in R(T^2)$.
Since $\lambda$ and $\mu$ commute, there exists a representation $\rho'$ such that $\rho'$ is conjugate to $\rho$ and both $\rho'(\lambda)$ and $\rho'(\mu)$ are upper triangular.
Considering the $(1,1)$-entries of these matrices, one can define the map $\theta\colon R(T^2) \to (\C^\times)^2/{\sim}$ by $\theta(\rho)=(\rho'(\lambda)_{11},\rho'(\mu)_{11})$, where 
$(L,M)\sim(L',M')$ 
if $L=L',\ M=M'$ or $L^{-1}=L',\ M^{-1}=M'$. 

It is easy to see that this map gives an identification $\theta\colon X(T^2) \to (\C^\times)^2/{\sim}$. 

The character variety of the complement $E(K)$ of a knot $K$ is complicated in general.
However, it is well known that if $K$ is a 2-bridge knot then $X(E(K))$ does not have an irreducible component of dimension larger than one.
More generally, if a 3-manifold $M$ contains no irreducible closed surface and $\partial M \cong T^2$, then $\dim C = 1$ for every irreducible component $C$ of $X(M)$ (see \cite[Section~2.4]{CCGLS94}).

\subsection{$A$-polynomial of knots}
\label{subsec:A-polynomial}
We briefly review the $A$-polynomial introduced by Cooper, Culler, Gillet, Long, and Shalen~\cite{CCGLS94} (see also \cite{CoLo96}) and a relation with the boundary slopes of knots.
For an oriented knot $K$, let $r\colon X(E(K)) \to X(\partial E(K))$ denote the regular map between affine algebraic sets induced by the inclusion and let $\pi\colon (\C^\times)^2 \to (\C^\times)^2/{\sim}$ be the natural projection. 
Here one takes $\lambda,\mu\in\pi_1(E(K))$ 
as a pair of a longitude $\lambda$ and a meridian $\mu$. 
We take $\lambda$ to be homologically trivial in $H_1(E(K);\Z)$. 
By using these $\lambda$ and $\mu$ one can also identify $\pi_1(\partial E(K))$ with $\Z^2$. 

For any $[\rho]\in X(E(K))$ 
one can take $[\rho']=[r(\rho)]$. 
To define the $A$-polynomial of a knot, we write 
$L$ for $\rho'(\lambda)_{11}$ and $M$ for $\rho'(\mu)_{11}$ as above. 

Then, the Zariski closure of $\pi^{-1}(\theta\circ r(X(E(K)))) \subset \C^2$ is an affine algebraic set whose irreducible components are curves $C_1,\dots,C_n$ and some points.
Since $\codim C_j = 1$, the ideal $I(C_j)$ is known to be principal, namely $I(C_j) = (f_j)$ for some $f_j \in \C[L,M]$.
It is known that there is $c \in \C$ such that $c f_1(L,M)\dotsm f_n(L,M) \in \Z[L,M]$ and its coefficients have no common divisor.
The \emph{$A$-polynomial} $A_K(L,M)$ of $K$ is now defined by $A_K(L,M) = c f_1(L,M)\dotsm f_n(L,M)$ up to sign, and it is independent of the choice of an orientation of $K$.

\begin{remark}
 Since $A_K(L,M)$ has the factor $L-1$ coming from abelian representations of $\pi_1(E(K))$, the $A$-polynomial is sometimes defined to be $A_K(L,M)/(L-1)$.
 This is not essential in our main theorem due to Lemma\nobreakspace \ref {lem:M-1}.
\end{remark}

\begin{lemma}\label{lem:M-1}
 If $\theta\circ r(\rho) = (L,1)$, then $L=1$.
 In particular, the $A$-polynomial $A_K(L,M)$ does not have the factor $M-1$.
\end{lemma}

\begin{proof}
 It follows from $r(\rho) = (L,1)$ that $\rho(\mu)$ is equal to the identity matrix $I_2$ or $
\begin{pmatrix}
 1 & 1 \\
 0 & 1
\end{pmatrix}
 $ up to conjugate.
 In the case $\rho(\mu)=I_2$, $\rho$ is trivial.
 In the latter case, $\rho(\lambda)$ is of the form $
\begin{pmatrix}
 1 & u \\
 0 & 1
\end{pmatrix}
 $ for some $u \in \C$, and hence $L=1$.
\end{proof}

We next see a relation between the $A$-polynomial and boundary slopes of $K$.
The rest of this subsection is devoted to proving Corollary\nobreakspace \ref {cor:gcd1_TwoBridge} which is used in Corollary\nobreakspace \ref {cor:FiniteImage_TwoBridge}, not in Theorem\nobreakspace \ref {thm:FiniteImage}.
Here, $p/q \in \Q\cup\{\infty\}$ is called a \emph{boundary slope} of $K$ if there exists a properly embedded incompressible surface $S$ in $E(K)$ such that $\partial S$ is parallel copies of a simple closed curve of slope $p/q$, namely the homology class of each boundary component of $S$ equals $p\mu+q\lambda \in H_1(E(K))$ up to sign.
We denote by $\BS(K)$ the set of boundary slopes of $K$.

For a polynomial $f(L,M) = \sum_{i,j}a_{ij} L^i M^j \in \Z[L,M]$, the \emph{Newton polygon} $N(f)$ of $f$ is defined by $N(f) = \Conv(\{(i,j) \in \Z^2 \mid a_{ij} \neq 0\})$, where $\Conv(T)$ denotes the convex hull of a subset $T$ in $\R^2$. 

We write by $\SS(P) \subset \Q\cup\{\infty\}$ the set of slopes of the sides of a polygon $P$.
Note that $\SS(N(f)) = \emptyset$ if and only if $f$ is a monomial.
The set $\SS(N(A_K))$ is closely related to $\BS(K)$.

\begin{theorem}[{\cite[Theorem~3.4]{CCGLS94}}]
\label{thm:CCGLS94}
 The inclusion $\SS(N({A_K})) \subset \BS(K)$ holds for every knot $K$.
\end{theorem}

Let us review some facts about the Minkowski sum.
For subsets $T$ and $U$ of $\R^2$, the \emph{Minkowski sum} $T+U$ is defined by $T+U = \{t+u \in \R^2 \mid t \in T,\ u \in U\}$.
One can see that $\Conv(T+U) = \Conv(T)+\Conv(U)$, and hence $N(fg) = N(f)+N(g)$.
The following proposition is well known and plays a key role in the next lemma.

\begin{proposition}[{see \cite[Section~15.1]{Gru67} for example}]
\label{prop:MinkowskiSum}
 Let $P$ and $Q$ be convex polygons.
 Then $\SS(P+Q) = \SS(P) \cup \SS(Q)$.
\end{proposition}

For a subset $S$ of $\Q\cup\{\infty\}$, we denote by $S^{-1}$ the set $\{s^{-1} \in \Q\cup\{\infty\} \mid s \in S\}$, where we use the convention $0 \cdot \infty = 1$.
Also, for a polynomial $f \in \Z[L,M]$, we define $f^T \in \Z[L,M]$ by $f^T(L,M)=f(M,L)$.

\begin{lemma}
\label{lem:SS^-1}
 Let $f_1, f_2 \in \Z[L,M]$.
 If $\SS(N(f_1)) \cap \SS(N(f_2))^{-1} = \emptyset$, then $\gcd(f_1, f_2^T)$ is a monomial.
\end{lemma}

\begin{proof}
 Let $g = \gcd(f_1, f_2^T)$.
 Then $g \mid f_1$ and $g^T \mid f_2$.
 By Proposition\nobreakspace \ref {prop:MinkowskiSum}, we have $\SS(N(g)) \subset \SS(N(f_1))$ and $\SS(N(g^T)) \subset \SS(N(f_2))$.
 Since $\SS(N(g)) = \SS(N(g^T))^{-1}$, the assumption implies that $\SS(N(g)) = \emptyset$, namely $g$ is a monomial.
\end{proof}

\begin{corollary}
\label{cor:gcd1_TwoBridge}
If $K_1$ and $K_2$ be any $2$-bridge knots, 
then it holds that $\gcd(A_{K_1}, A_{K_2}^T) = 1$.
\end{corollary}

\begin{proof}
 By \cite[Theorem~1(b)]{HaTh85}, $\BS(K_i) \subset 2\Z$ holds.
 It follows from Theorem\nobreakspace \ref {thm:CCGLS94} that $\SS(N(A_{K_1})) \cap \SS(N(A_{K_2}))^{-1} = \emptyset$, and hence $\gcd(A_{K_1}, A_{K_2}^T)$ is a monomial by Lemma\nobreakspace \ref {lem:SS^-1}.
 Here, in general, the $A$-polynomial of a knot $K$ is divided by neither $L$ nor $M$ by definition.
 Therefore, the monomial must be 1.
\end{proof}

\subsection{The $\SL(2,\C)$-Reidemeister torsion of 3-manifolds} 
\label{subsec:ReidemeisterTorsion}

For precise definitions of a Reidemeister torsion, 
please see Johnson~\cite{Joh}, Kitano~\cite{Kit94a, Kit94b} and Milnor~\cite{Mil61,Mil62} as references. 

Let $M$ be a 3-manifold and let $\rho \in R(M)$ be an acyclic representation. 
That is, $C_\ast(M;\C^2_{\rho})$ is an acyclic chain complex with twisted coefficients. 

Then one gets a nonzero complex number $\tau_\rho(M) \in \C^\times$ for an acyclic chain complex $C_\ast(M;\C^2_\rho)$. 
We call it the \emph{$\SL(2,\C)$-Reidemeister torsion} of $M$ for $\rho$.

\begin{remark}
 Throughout this paper, we set $\tau_\rho(M) = 0$ if $\rho$ is not acyclic.
 Then $\tau_\rho(M)$ can be regarded as a function on $R(M)$ and also on $X(M)$.
\end{remark}

One can use the well-known multiplicativity of the Reidemeister torsion to compute it as below. 

\begin{proposition}
\label{prop:Multiplicativity}
Let $M$ be a $3$-manifold decomposed into $M_1$ and $M_2$ by an embedded torus $T^2$. 
Let $\rho\colon \pi_1(M) \to\SL(2,\C)$ be a representation. 
Suppose that $\rho$ is acyclic on $\pi_1(T^2)$. 
Then it holds that $\rho$ is acyclic on $\pi_1(M)$ if and only if it is acyclic on both $\pi_1(M_1)$ and $\pi_1(M_2)$. 
Further in this case it holds that 
 $$\tau_\rho(M) = \tau_\rho(M_1) \tau_\rho(M_2).$$
\end{proposition}

One needs the acyclicity of representations to use the above. 
First we mention the following lemma. 

\begin{lemma}
\label{lem:AcyclicOnTorus}
Let $\rho$ be a representation $\pi_1(T^2) \to \SL(2,\C)$. 
Then it holds that $\rho$ is acyclic if and only if $\rho$ is not parabolic. 
Here $\rho$ is said to be \emph{parabolic} if $\tr\rho(x)=2$ for any $x \in\pi_1(T^2)$.
\end{lemma}

\begin{proof}

First note that for a basis $\{x, y\}$ of $\pi_1(T^2)$ the chain complex $C_\ast(T^2;\C^2_\rho)$ is given by 
\[
0 \to \C^2 \xrightarrow{\partial_2} \C^2\oplus \C^2 \xrightarrow{\partial_1} \C^2 \to 0,
\]
where 
\[
\partial_2=\begin{pmatrix} -(\rho(y)-I_2)& \rho(x)-I_2\end{pmatrix},\ 
\partial_1=\begin{pmatrix} \rho(x)-I_2\\ \rho(y)-I_2\end{pmatrix}.
\]
We here show that $\rho$ is not parabolic if and only if $H_0(T^2;\C^2_\rho)=0$.
If $\rho$ is not parabolic, then there is a basis $\{x, y\}$ such that $\det(\rho(x)-I_2)\neq 0$, and thus $H_0(T^2;\C^2_\rho)=0$.
Conversely, if $\rho$ is parabolic, then $\rho(x)$ and $\rho(y)$ are simultaneously of the form $
\begin{pmatrix}
 1 & \ast \\
 0 & 1
\end{pmatrix}$
by taking conjugate, and therefore $H_0(T^2;\C^2_\rho) \neq 0$.

Next, if $H_0(T^2;\C^2_\rho)=0$, then $\rho$ is acyclic.
Indeed, by Kronecker duality (or the universal coefficient theorem) and Poincar\'e duality, $H_2(T^2;\C^2_\rho) \cong H_0(T^2;\C^2_{\check{\rho}})$, where $\check{\rho}(\gamma) := {}^t\!\rho(\gamma)^{-1}$.
When $\rho$ is parabolic, so is $\check{\rho}$.
It follows from $\chi(T^2)=0$ that $H_1(T^2;\C^2_\rho)=0$.
\end{proof}

\section{Proof of the main theorem}
\label{sec:MainTheorem}

Recall that $\Sigma(K_1,K_2)$ denotes the splice. 
The following lemma is shown in \cite[Proof of Corollary~3.3]{BoCu08}. 
We give a proof to be self-contained. 

\begin{lemma}
\label{lemma:IrreducibilityOnSplice}
If $\rho$ is irreducible on $\pi_1(\Sigma(K_1,K_2))$, 
then the restrictions of $\rho$ on $\pi_1(E(K_1))$ and $\pi_1(E(K_2))$ are also irreducible. 
\end{lemma}

\begin{proof}
Assume that $\rho$ is reducible on $\pi_1(E(K_1))$.  
Then we may take $\rho$ as an upper triangular representation on it. 
Since the longitude $\lambda_1$ of $K_1$ belongs to the commutator subgroup $[\pi_1(E(K_1)),\pi_1(E(K_1))]$, 
then one can see that $L_1$ is an upper triangular parabolic matrix as 
$L_1= \rho(\lambda_1) = \begin{pmatrix} 1 & \alpha\\ 0 & 1\end{pmatrix}$.

If $\alpha=0$, then $L_1$ is the identity and hence $X_2=L_1$ is also the identity matrix. 
This means that $\rho$ must be trivial on $\pi_1(E(K_2))$ and this is a contradiction. 

Therefore we may assume $\alpha\neq 0$. 
Since $X_1$ commutes with $L_1$, 
then $X_1$ is also an upper triangular matrix as  
$X_1=\begin{pmatrix} \pm 1 & \beta\\ 0 & \pm 1\end{pmatrix}~(\beta\neq 0)$. 
Hence the image $\rho(\pi_1(E(K_1)))$ is an upper triangular subgroup.  
Since this is an abelian subgroup in $\SL(2,\C)$, then $L_1$ must be also the identity. 
This is a contradiction. 
\end{proof}

\begin{remark}
By the above arguments, it can be seen that there exists no reducible representation except the trivial representation. 
\end{remark}

Next we can see the following. 
\begin{proposition}
\label{prop:AcyclicOnSplice}
If $\rho\colon \pi_1(\Sigma(K_1,K_2)) \to\SL(2,\C)$ be an acyclic representation, 
then its restriction $\rho|_{\pi_1(T^2)}$ is also acyclic. 
\end{proposition}

\begin{proof}
Assume that $\rho|_{\pi_1(T^2)}$ is not acyclic. 
Consider the homology long exact sequence for 
\[
0 \to C_\ast(T^2;\C^2_\rho) \to C_\ast(E(K_1);\C^2_\rho)\oplus C_\ast(E(K_2);\C^2_\rho) \to C_\ast(\Sigma(K_1,K_2);\C^2_\rho) \to 0.
\]
Here we simply write $\rho$ for each of $\rho|_{\pi_1(T^2)}$, $\rho|_{\pi_1(E(K_1))}$, and $\rho|_{\pi_1(E(K_2))}$. 
Since $C_\ast(\Sigma(K_1,K_2);\C^2_\rho)$ is acyclic, we have the exact sequences
\[
0 \to H_2(T^2;\C^2_\rho) \to H_2(E(K_1);\C^2_\rho)\oplus H_2(E(K_2);\C^2_\rho) \to 0,
\]
\[
0 \to H_1(T^2;\C^2_\rho) \to H_1(E(K_1);\C^2_\rho)\oplus H_1(E(K_2);\C^2_\rho) \to 0,
\]
\[
0 \to H_0(T^2;\C^2_\rho) \to H_0(E(K_1);\C^2_\rho)\oplus H_0(E(K_2);\C^2_\rho) \to 0.
\]
Since $\rho$ is not acyclic on $\pi_1(T^2)$, 
$\rho$ is parabolic on it 
by Lemma\nobreakspace \ref {lem:AcyclicOnTorus}. 
If it is trivial on $\pi_1(T^2)$, it should be trivial on $\pi_1(\Sigma(K_1,K_2))$. 
Then it is not acyclic on $\Sigma(K_1,K_2)$. 
For any non-trivial parabolic representation $\rho$ on $\pi_1(T^2)$, 
it is easy to see 
\[
H_2(T^2;\C^2_\rho)\cong H_0(T^2;\C^2_\rho)\cong \C,\ H_1(T^2;\C^2_\rho)\cong \C^2
\]
by the proof of Lemma\nobreakspace \ref {lem:AcyclicOnTorus}.
If $\rho$ is irreducible, then both $\rho|_{\pi_1(E(K_1))}$ and $\rho|_{\pi_1(E(K_2))}$ are irreducible 
by Lemma\nobreakspace \ref {lemma:IrreducibilityOnSplice}.  
Then it holds that $H_0(E(K_1);\C^2_\rho)$ and $H_0(E(K_2);\C^2_\rho)$ are vanishing. 
Therefore $H_0(T^2;\C^2_\rho)$ is vanishing in this case by the above exact sequences. 
It is contradiction. 

Next assume that $\rho$ is reducible. 
Now we may assume that the image of $\rho$ belongs to the upper triangular subgroup. 
It is easily seen that the images of the longitudes are trivial $I_2$ or $\begin{pmatrix} 1 & \ast \\0 & 1\end{pmatrix}$ 
since the longitudes belong to the commutator subgroup. 
Therefore the image of each meridian is in the upper triangular parabolic subgroup by the definition of a splice, and thus $\rho$ is abelian. 
This contradicts the fact that the abelianization of $\pi_1(\Sigma(K_1,K_2))$ is trivial.
\end{proof}

\begin{lemma}
 \label{lem:dim1}
 Let $f\colon X \to Y$ be a non-constant regular map between affine algebraic sets $X$ and $Y$.
 If $X$ is irreducible and $\dim X = 1$, then $f^{-1}(\{y\})$ is a finite \textup{(}possibly empty\textup{)} set for any $y \in Y$.
\end{lemma}

\begin{proof}
 The inverse image $f^{-1}(\{y\})$ is a closed subset of $X$, namely $f^{-1}(\{y\})$ is a finite union of irreducible algebraic sets.
 Since they are proper algebraic subsets of $X$, they are of dimension zero.
\end{proof}

The next lemma follows from Lemma\nobreakspace \ref {lem:dim1} or B\'ezout's theorem. 

\begin{lemma}
 \label{lem:gcd1} 
 Let $f, g \in \C[L, M]$.
 Then $\{f=g=0\} \subset \C^2$ is a finite set if and only if $\gcd(f,g)=1$.
\end{lemma}


Using the above lemmas and propositions, we prove the main theorem.

\begin{proof}[Proof of Theorem\nobreakspace \ref {thm:FiniteImage}]
 First note that $\gcd_{\Z[L,M]}(f,g) = \gcd_{\C[L,M]}(f,g)$ holds for $f, g \in \Z[L,M]$ up to multiplication by elements of $\C^\times$.
 By Lemma\nobreakspace \ref {lem:gcd1}, the intersection 
 $$
 \{(L,M) \in \C^2 \mid A_{K_1}(L,M)=A_{K_2}^T(L,M)=0\}
 $$
 of the algebraic curves defined by $A_{K_1}$ and $A_{K_2}^T$ is a finite set $A$.
 Let us prove that the image of $X(\Sigma(K_1,K_2)) \to X(E(K_i))$ is a finite set $X_i$ for $i=1,2$.
 Then Propositions\nobreakspace \ref {prop:Multiplicativity} and\nobreakspace  \ref {prop:AcyclicOnSplice} complete the proof.
 
 By the definition of the $A$-polynomial, $\theta\circ r_i(X_i) \subset A$.
 It follows from Lemma\nobreakspace \ref {lem:dim1} and the second condition in Theorem\nobreakspace \ref {thm:FiniteImage} that $r_i^{-1}(\theta^{-1}(A))$ is a finite set.
 Thus, $X_i$ is also a finite set.
%
\end{proof}

We next prove Corollary\nobreakspace \ref {cor:FiniteImage_TwoBridge}.
Let $K$ be a $2$-bridge knot. 
Take and fix a presentation of $\pi_1(E(K))$ and write $\phi(s,t)$ to its Riley polynomial \textup{(}see Section\nobreakspace \ref {sec:Computation}\textup{)}.
Then the following lemma is a consequence of \cite[Lemma~2]{Ril84}.

\begin{lemma}
 \label{lem:Riley}
 The coefficient of the leading term of $\phi(s,t) \in \Z[s^{\pm 1}, t]$ with respect to $t$ is a monomial of $s$.
\end{lemma}

\begin{proof}[Proof of Corollary\nobreakspace \ref {cor:FiniteImage_TwoBridge}]
 It suffices to check that any pair of 2-bridge knots $K_1$ and $K_2$ satisfies the conditions in Theorem\nobreakspace \ref {thm:FiniteImage}.
 First, Corollary\nobreakspace \ref {cor:gcd1_TwoBridge} implies $\gcd(A_{K_1}(L,M), A_{K_2}(M,L)) = 1$.
 Let $C$ be an irreducible component of $X(E(K_i))$.
 
 If $C$ consists of reducible representations, then $\dim C = 1$ and $r_i(C) \subset X(\partial E(K_i))$ is not a point.
 Otherwise, $C$ is described by an irreducible factor of the Riley polynomial of $K_i$, and hence $\dim C = 1$.
 Assume that $r_i(C)$ is a point $\theta^{-1}(L,M)$.
 Then the function $\tr\rho(\mu)$ is the constant $M+M^{-1}$ on $C$, and thus $s-M \mid \phi(s,t)$.
 Since $M \neq 0$, this contradicts Lemma\nobreakspace \ref {lem:Riley}.
\end{proof}


We put the following problem.
\begin{problem}
When $\RT(\Sigma(K_1,K_2))$ is an infinite set?
Or is it always a finite set?
\end{problem}

Here we give an observation when $\dim C >1$ in Theorem\nobreakspace \ref {thm:FiniteImage}. 

\begin{example}
\label{example:Montesinos}

Let $K$ be the Montesinos knot $M(1/3,1/3,1/3,1/3,1/2)$ (see Figure\nobreakspace \ref {fig:PaPo13}).
Then $\pi_1(E(K))$ has the presentation
\[
\ang{
\mu_1,\dots,\mu_5 \biggm| \parbox{23em}{$
\mu_{i}\mu_{i+1}^{-1}\mu_{i}^{-1}
\mu_{i+1}
\mu_{i}^{-1}=\mu_{i+1}\mu_{i+2}^{-1}\mu_{i+1}\mu_{i+2}^{-1}\mu_{i+1}^{-1}\mu_{i+2}\mu_{i+1}^{-1} \\
\mu_4\mu_5^{-1}\mu_4^{-1}\mu_5\mu_4^{-1}=\mu_5\mu_1^{-1}\mu_5\mu_1\mu_5^{-1} 
\quad (i=1,2,3)
$}
}.
\]
Note that $\mu_1$ is conjugate to $\mu_2^{-1}$, $\mu_3$, $\mu_4^{-1}$ and $\mu_5$.
It follows from \cite[Theorem~1]{PaPo13} that there is an irreducible component of $X(E(K))$ with $\mathrm{dim} \geq 2$.
In fact, we construct a 2-parameter family $C$ of representations by a bending (see Section~\ref{sec:Computation}) along the sphere $S$ intersecting $K$ at 4 points illustrated in Figure~\ref{fig:PaPo13}. 

For $s \in \C^\times\setminus\{1\}$, we first define the representation 
$\rho_s\colon \pi_1(E(K)) \to \SL(2,\C)$ by 
$\rho_s(\mu_j)=
\begin{pmatrix}
 s^{-1} & 0 \\
 s^2-1+s^{-2} & s
\end{pmatrix}$
if $j=1,3,5$ and $\rho_s(\mu_j)=
\begin{pmatrix}
 s & 1 \\
 0 & s^{-1}
\end{pmatrix}$
if $j=2,4$.
Note that $\rho_s$ factors through 
$\pi_1(E(\bar{3}_1))=\ang{ x,y \mid xyx=yxy }$, 
where $\bar{3}_1$ denotes the right-handed trefoil knot. 
That is, $\rho_s$ comes from 
$\bar{\rho}_s\colon \pi_1(E(\bar{3}_1)) \to \SL(2,\C)$ by 
$\bar{\rho}_s(x)=\rho_s(\mu_j)$ if $j=1,3,5$ 
and 
$\bar{\rho}_s(y)=\rho_s(\mu_j)^{-1}$ if $j=2,4$. 
Here we have $\rho_s(\mu_1')=\rho_s(\mu_1)$ since $\rho_s(\mu_1)=\rho_s(\mu_5)$.
By finding the tangle enclosed by the dotted circle drawn in Figure~\ref{fig:PaPo13} which is a part of $\bar{3}_1$, 
we can see that the restriction of $\rho_s$ to $\pi_1(S\setminus K)$ is invariant under conjugation by
\[
P_u=
\begin{pmatrix}
 \left(\frac{s^2-1+s^{-2}}{u}\right)^{1/2} & \frac{\left(\frac{s^2-1+s^{-2}}{u}\right)^{1/2} - \left(\frac{s^2-1+s^{-2}}{u}\right)^{-1/2}}{s-s^{-1}} \\
 0 & \left(\frac{s^2-1+s^{-2}}{u}\right)^{-1/2}
\end{pmatrix},
\]
where $u \in \C^\times$ (see Lemma~\ref{lem:Aform}).
Therefore, one obtains representations $\rho_{s,u}\colon \pi_1(E(K)) \to \SL(2,\C)$ by
\[\rho_{s,u}(\mu_j)=
\begin{cases}
 \rho_s(\mu_j) & \text{if $j=3$,} \\
 P_u\rho_s(\mu_j)P_u^{-1} & \text{if $j=1,2,4,5$.}
\end{cases}
\]

By the above bending construction, the set $C=\{\rho_{s,u}\}$ is still a 2-parameter family in $X(E(K))$. 
Now one can also check it directly 
$$
\tau_{\rho_{s,u}}(E(K)) 
 = \frac{144s^{-4}(s-1)^8}{-s^{-1}(s-1)^2} = -144(\tr\rho_{s,u}(\mu_2) -2)^{3},
$$
and hence $\tau_{\rho_{s,u}}(E(K))$ depends only on $\tr\rho_{s,u}(\mu_2) = s+s^{-1}$.

On the other hand, to be independent of $u$, 
it can be explained by the generalized multiplicativity of the Reidemeister torsion 
to the decomposition 
$E(K) = M_1 \cup_{S_0} M_2$ along the surface $S_0 = S\cap E(K)$ with 4 boundary components. 
Although $M_1$, $M_2$ and $S_0$ are not acyclic, 
after fixing suitable bases of $H_1(M_1;\C^2_{\rho_{s,u}})$, $H_1(M_2;\C^2_{\rho_{s,u}})$ 
and $H_1(S_0;\C^2_{\rho_{s,u}})$, 
we obtain 
$\tau_{\rho_{s,u}}(E(K)) = \tau_{\rho_{s,u}}(M_1)\tau_{\rho_{s,u}}(M_2)/\tau_{\rho_{s,u}}(S_0)$. 
By the construction of $\rho_{s,u}$, 
we see that the value of the right-hand side is independent of $u$. 


Let $RT_C$ be the subset of $\RT(\Sigma(K,K))$ consisting of $\tau_\rho(\Sigma(K,K))$'s where the restriction of $\rho$ to each $E(K)$ belongs to $C$.
Then $RT_C$ is a finite set even though $C \subset X(\Sigma(K,K))$ is 2-dimensional.
Indeed, one can check that $\tr\rho_{s,u}(\lambda) = s^{24}+s^{-24}$, and thus there are finitely many solutions $(s_1,s_2)$ of $\tr\rho_{s_1,u_1}(\mu) = \tr\rho_{s_2,u_2}(\lambda)$ and $\tr\rho_{s_1,u_1}(\lambda) = \tr\rho_{s_2,u_2}(\mu)$.
We conclude that there are finitely many possibilities of the value $\tau_\rho(\Sigma(K,K)) = \tau_{\rho_{s_1,u_1}}(E(K))\tau_{\rho_{s_2,u_2}}(E(K))$.
\end{example}

\begin{problem}
Can we relax the assumption ``either $\dim C = 0$, or $\dim C = 1$ and its image under the map $X(E(K_i)) \to X(\partial E(K_i))$ is not a point'' in Theorem\nobreakspace \ref {thm:FiniteImage}?
\end{problem}

\begin{figure}
 \centering
 \includegraphics[width=0.6\textwidth]{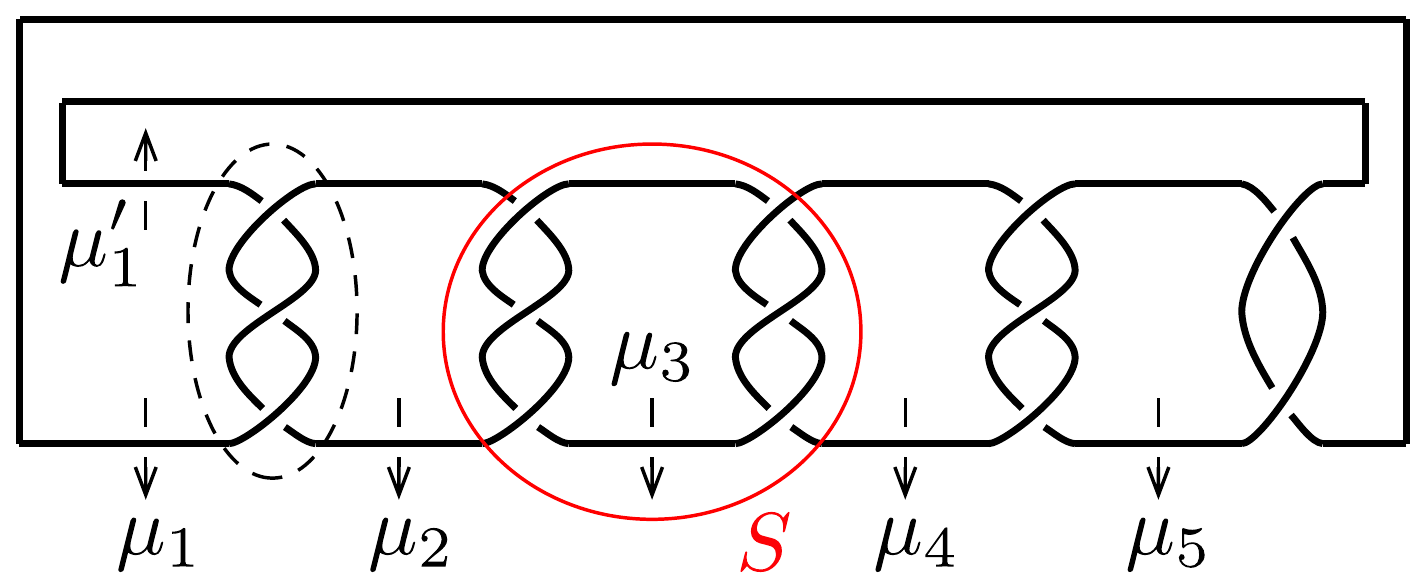}
 \caption{The Montesinos knot $K=M(1/3,1/3,1/3,1/3,1/2)$. 
 }
 \label{fig:PaPo13}
\end{figure}

\section{Computational observation}
\label{sec:Computation}

\subsection{Concrete description of $X^\irr(\Sigma(K_1,K_2))$} 

In this section we observe some examples of splices.
We use Mathematica to compute matrices.
Recall that $X^\irr(M)$ is identified with $R^\irr(M)/\SL(2,\C)$. 
The construction of deformations of a representation used in this section is called a \emph{bending construction} or simply a \emph{bending}. 
See \cite{JoMi87, Mot88} as a reference. 

Here we compute $X^\irr(\Sigma(K,K))$ for the trefoil knot and the figure-eight knot $K$ by using a presentation of a twist knot. 
Let $J(2,2q)$ be a twist knot where $q$ is a nonzero integer. 
Please see \cite{Mor08} as a reference for twist knots. 

A presentation of $\pi_1(E(J(2,2q)))$ is given as 
\[
\pi_1(E(J(2,2q)))= \ang{x, y\mid z^q x=y z^q},\ z=[y,x^{-1}]
\]
We take a representation $\rho\colon \ang{x, y} \to \SL(2,\C)$ 
from the free group $\ang{x, y}$ in $\SL(2,\C)$ by the correspondence 
\[
\rho(x)=\begin{pmatrix}
s & 1\\
0 & 1/s
\end{pmatrix},\ 
\rho(y)=\begin{pmatrix}
s & 0\\
-t & 1/s
\end{pmatrix}
\quad(s,t \in \C^\times).
\]

We use a small letter for a group element and its capital letter for the image of a small letter, like $X$ for $\rho(x)$. 
For $\rho(z^q)=Z^q$, 
we put the matrix $Z^q=\begin{pmatrix} z_{11} & z_{12}\\z_{21} & z_{22}\end{pmatrix}$. 

We define the Riley polynomial to be $\phi_q(s,t)=z_{11}+(1/s-s)z_{12}$. 
It can be checked that the previous representation gives an irreducible representation of $\pi_1(E(J(2,2q)))$ in $\SL(2,\C)$ if and only if $(s,t)$ satisfies $\phi_q(s,t)=0$. 

It is seen that 
any $[\rho]\in X(E(J(2,2q)))$ can be parametrized by 
\[
\xi=\tr\rho(x) = \tr\rho(y) =s+1/s,
\]
\[
\tr\rho(xy) = s^2+1/s^2-t = (s+1/s)^2-t-2=\xi^2-t-2,
\]
and then by $\xi$ and $t$. 

Here we take other words $\tilde{z}=[x,y^{-1}]$ and $\lambda=\tilde{z}^q z^q$. 
This $x$ gives a meridian of $J(2,2q)$ and this $\lambda$ does the corresponding longitude for $x$. 
Here $\lambda$ is homologically trivial. 
Therefore $\ang{x, \lambda}$ is the free abelian group of rank $2$ and $\rho(x)=X$ commutes with $\rho(\lambda)=L$. 
We can find another matrix which commutes with $X$ and $L$ by direct computations. 

\begin{lemma}
\label{lem:Aform}
Any matrix $A$ which commutes with 
$X=\begin{pmatrix} s & c^2 \\
 0 & 1/s  
\end{pmatrix}
\ (s\neq \pm 1,~c\neq 0)$ has a form of 
\[
A=\begin{pmatrix}
a & \frac{a-1/a}{s-1/s} c^2\\
 0 & 1/a \\
\end{pmatrix}
\]
for some $a \in \C^\times$.
\end{lemma}

Now we consider two copies $K_1,K_2$ of $J(2,2q)$ and 
\[
\pi_1(E(K_1))=\ang{x_1,y_1\mid z_1^q x_1=y_1 z_1^q},\ z_1=[y_1,x_1^{-1}],
\]
\[
\pi_1(E(K_2))=\ang{x_2,y_2\mid z_2^q x_2=y_2 z_2^q},\ z_2=[y_2,x_2^{-1}].
\]
Further 
\[
\begin{split}
\pi_1(\Sigma(K_1,K_2))
&=\pi_1(E(K_1))\ast_{\pi_1(T^2)}\pi_1(E(K_2))\\
&=\ang{x_1,y_1,x_2,y_2 \mid z_1^q x_1=y_1 z_1^q,\ z_2^q x_2=y_2 z_2^q,\ x_1=\lambda_2,\ \lambda_1=x_2}.
\end{split}
\]

We consider an irreducible representation $\rho\colon \pi_1(\Sigma(K_1,K_2)) \to \SL(2,\C)$. 
Up to conjugate, we can set that 
\[
X_1=\rho(x_1)=\begin{pmatrix}
s_1 & 1\\
0 & 1/s_1
\end{pmatrix},\ 
Y_1=\rho(y_1)=\begin{pmatrix}
s_1 & 0\\
-t_1 & 1/s_1
\end{pmatrix}.
\]

First note that we treat cases of $s_1\neq \pm 1$.
Further we may assume that 
$X_2$ is conjugate to $\begin{pmatrix}
s_2 & 1\\
0 & 1/s_2
\end{pmatrix}$, 
and $Y_2$ to $\begin{pmatrix}
s_2 & 0\\
-t_2 & 1/s_2
\end{pmatrix}$ simultaneously, 
as 
\[
X_2=H\begin{pmatrix}
s_2 & 1\\
0 & 1/s_2
\end{pmatrix}H^{-1},\ 
Y_2=H\begin{pmatrix}
s_2 & 0\\
-t_2 & 1/s_2
\end{pmatrix}H^{-1}
\]
for some $H\in\SL(2,\C)$. 

Here we require the conditions 
\[
X_1=L_2,\ L_1=X_2\] 
to get a representation on $\pi_1(\Sigma(K_1,K_2))$. 
It can be seen that $L_1$ is an upper triangular matrix and then $X_2$ is also an upper triangular matrix. 
By taking 
\[
H=\begin{pmatrix}
c & 0\\
0 & 1/c
\end{pmatrix}\ (c\neq 0),
\]
one has
\[
\begin{split}
X_2&=H\begin{pmatrix}
s_2 & 1\\
0 & 1/s_2
\end{pmatrix}H^{-1}=\begin{pmatrix}
s_2 & c^2\\
0 & 1/s_2
\end{pmatrix},\\ 
Y_2&=H\begin{pmatrix}
s_2 & 0\\
-t_2 & 1/s_2
\end{pmatrix}H^{-1}
=\begin{pmatrix}
s_2 & 0\\
-t_2/c^2 & 1/s_2
\end{pmatrix}.
\end{split}
\]

Here $L_2$ is also an upper triangular matrix and $L_2=X_1$.
Now any $[\rho]=[\rho_1 \ast\rho_2]\in X(\Sigma(K_1,K_2))$ is corresponding to 
$(X_1,Y_1,X_2,Y_2)=(X_1,Y_1,L_1,Y_2)$ 
of the above forms. 
For $a \in \C^\times$ we define $A_a$ by
\[
A_a=\begin{pmatrix}
a & \frac{a-1/a}{s_1-1/s_1} \\
 0 & 1/a \\
\end{pmatrix}
\]
and now consider deformations $[\rho_a]=[(A_a\rho_1 A_a^{-1})\ast\rho_2]$ of $[\rho]=[\rho_1 \ast\rho_2]$ as
\[
(A_a X_1 A_a^{-1},A_a Y_1 A_a^{-1},X_2,Y_2 )
=(X_1,A_a Y_1 A_a^{-1},X_2, Y_2 ).
\]

\begin{lemma}
It holds that $A_a L_1 A_a^{-1}=L_1$.
\end{lemma}
\begin{proof}
We prove $A_a L_1 A_a^{-1}=L_1$. 
We may assume $s\neq \pm 1$. 
Here we take eigenvectors 
${\bf u}_1=\begin{pmatrix}1\\ 0\end{pmatrix},{\bf u}_2\in\C^2$ for $X_1$ 
such that $X_1{\bf u}_1=s_1 {\bf u}_1,X_1{\bf u}_2=s_1^{-1}{\bf u}_2$.  
Since $X_1 L_1=L_1 X_1$, 
one has 
\[
\begin{split}
X_1 L_1 {\bf u}_2&=L_1 X_1 {\bf u}_2\\
&=L_1 s_1^{-1}{\bf u}_2\\
&=s_1^{-1}L_1{\bf u}_2.
\end{split}
\]
Hence there exists a nonzero constant $\gamma$ such that $L_1{\bf u}_2=\gamma {\bf u}_2$. 
This means that $L_1$ has also ${\bf u}_1,{\bf u}_2\in\C^2$ as eigenvectors. 
By similar arguments for $A_a X_1=X_1 A_a$, 
one sees $A_a$ has ${\bf u}_1,{\bf u}_2\in\C^2$ as eigenvectors. 
Therefore it is seen that $X_1,L_1,A_a$ are simultaneously diagonalizable and in particular $A_a L_1 A_a^{-1}=L_1$. 
\end{proof}

By the above lemma, one can see 
$A_a \rho_1 A_a^{-1}=\rho_1$ 
on the subgroup $\pi_1(T^2)$ generated by $\{x_1,l_1\}=\{x_2,l_2\}$ 
and then 
$\rho_{a}=(A_a \rho_1 A_a^{-1})\ast \rho_2 $ gives an irreducible representation of $\pi_1(\Sigma(K_1,K_2))$. 

Further if $a\neq 1$, then 
$A_a Y_1 A_a^{-1}\neq Y_1$. 
This implies
$
\rho_{a}\neq \rho
\in R(\Sigma(K_1,K_2))  .
$
It can be seen by the following computations. 
First one sees that
\[
\begin{split}
\tr(\rho_1\ast\rho_2(y_1 x_2 ))
&
=\tr(Y_1 X_2 )\\
&=s_1s_2+\frac{1}{s_1 s_2}-c^2 t_1.
\end{split}
\]

On the other hand,  one sees that 
\[
\begin{split}
& \tr\left((A_a\rho_{1}A_a^{-1})\ast \rho_{2})(y_1 x_2 )\right) \\
=& \tr(A_a Y_1 A_a^{-1} X_2 ) \\
=& s_1 s_2 + \frac{1}{s_1 s_2}
 +
 \left\{
 \frac{ (s_2 -\frac{1}{s_2})}{(s_1 -\frac{1}{s_1})}  \left(\frac{1}{a^2}-1\right) - \frac{c^2}{a^2}
 \right\}
 t_1.
\end{split}
\]
Therefore we can find one character
\[
[\rho]\mapsto\tr\rho(y_1 x_2 )
\]
which is not constant on $X(\Sigma(K_1,K_2))$ and 
we know $X(\Sigma(K_1,K_2))$ has at least one dimension near $[\rho]$. 


\begin{proposition}
$X(\Sigma(K_1,K_2))$ has just one dimension near $[\rho]$.
\end{proposition}

\begin{proof}
Take and fix any $[\rho]=[\rho_1\ast\rho_2]\in X(\Sigma(K_1,K_2))$. 
It is seen that the character variety $X(\Sigma(K_1,K_2))$ has at least one dimension near $[\rho]$ 
by a  bending construction. 

Consider another one parameter family 
\[
\{[\rho_u]\}_u = \{[\rho_{1,u}\ast\rho_{2,u}]\}_u
\subset X(\Sigma(K_1,K_2))
\]
such that $[\rho_u]=[\rho]$. 
Here recall there exist only finitely many quadruples 
$\{(s_1,t_1,s_2,t_2)\text{'s}\}$ 
for this fixed $[\rho]=[\rho_1\ast\rho_2]\in X(\Sigma(K_1,K_2))$ by the proof of the main theorem. 
Then we may assume that 
$[\rho_u]=[\rho_{1,u}\ast\rho_{2}]$ and $[\rho_{1,u}]=[\rho_1]\in X(K_1)$ 
for any $u$. 
Hence one gets $[\rho_u]=[\rho_{1,u}\ast\rho_{2}]=[(B \rho_{1} B^{-1})\ast\rho_{2}]$, where $B\in \SL(2,\C)$. 
Because $B$ must commute with $X_1$ and $L_1$, 
then $B$ has a similar form as $A_a$ in Lemma\nobreakspace \ref {lem:Aform}. 
Therefore this is a bending construction and the dimension of deformations is one. 
\end{proof}

\subsection{$q=1$ Case}

Here we put $q=1$. This $J(2,2)$ is the trefoil knot. 
We write again
\[
X=
\begin{pmatrix} s & 1 \\
 0 & 1/s  
\end{pmatrix},\ 
Y=
\begin{pmatrix}
 s & 0 \\
 -t & 1/s 
 \end{pmatrix}
\]
and 
\[
\begin{split}
Z&=[Y,X^{-1}]\\
&=\begin{pmatrix}
1-s^2 t &\frac{1}{s}-s (1+t) \\
-\frac{t}{s}+s
t(1+t) & 1+(2-\frac{1}{s^2}) t+t^2
\end{pmatrix},\\
\tilde{Z}
&=[X,Y^{-1}]\\
&=\begin{pmatrix}
1-\left(-2+s^2\right) t+t^2 & \frac{-1+s^2-t}{s}\\
\frac{t\left(1-s^2+t\right)}{s} & 1-\frac{t}{s^2}
\end{pmatrix},\\
ZX-YZ
&=\begin{pmatrix}
0 & -1+1/s^2+s^2-t\\
s(-1+1/s^2+s^2-t)  & 0
\end{pmatrix}.
\end{split}
\]

The condition that $(s,t)$ gives a representation is 
$-1+1/s^2+s^2-t=0$. 
On the other hand, 
\[\begin{split}
\phi_1(s,t)
&=w_{11}+(1/s-s)w_{12} \\
&=1-s^2 t +(1/s-s)(1/s-s(1+t)) \\
&=1-s^2 t +1/s^2-1-t-1+s^2+s^2 t \\
&=-1+1/s^2+s^2-t \\
&=\xi^2-3-t,
\end{split}\]
where $m=s+1/s$. 

Hence in the case of the trefoil knot, 
one sees  
\[
t=\xi^2-3
\]
and $X^\irr(E(J(2,2)))$ is given by 
\[
\{(\xi,t)\in\C^2\mid t=\xi^2-3,\ t\neq 0\}.
\]

If $t=0$, then the corresponding representation is not irreducible. 

\begin{remark}
If $s=1$, that is, $\xi=2$, 
then the chain complex is not acyclic. 
In the other cases, $\tau_\rho(E(J(2,2)))=2$.
\end{remark}

Compute 
\[
\begin{split}
L
&=\tilde{Z}Z\\
&=\begin{pmatrix}
1-t^2+s^4 t^2-t^3+\frac{t (1+t)}{s^2}-s^2 t \left(1+t+t^2\right) 
& \frac{\left(1+s^2\right)
t\left(1+t+s^4 (1+t)-s^2 \left(3+3 t+t^2\right)\right)}{s^3}\\
\frac{t^2
\left(1+s^6-s^2 t-s^4 t\right)}{s^3} & 
1-t^2+\frac{t^2}{s^4}-t^3+s^2 t (1+t)-\frac{t \left(1+t+t^2\right)}{s^2}
\end{pmatrix}.
\end{split}
\]

By putting $t=1-(1/s^2+s^2)$,  
one obtains
\[
L=
\begin{pmatrix}
-s^6 & -\frac{(1+s^2+s^4)(1+s^6)}{s^5}\\
0 & -1/s^6
\end{pmatrix}
\]
and 
\[
\tr(L)=-s^6-1/s^6=-T_6(m).
\]
Here $T_6(x)=x^6-6x^4+9x^2-2$ is the normalized Chebyshev polynomial of degree $6$. 
Remark that $T_6(x)$ has the property $T_6(2\cos\theta)=2\cos 6\theta$.

By relations $x_1=\lambda_2,\ \lambda_1=x_2$, 
one has 
\[
\tr(X_1)=\tr(L_2),\ \tr(L_1)=\tr(X_2).
\]
By putting $\xi_1=s_1+1/s_1,\ \xi_2=s_2+1/s_2$, 
one obtains 
\[
\xi_1=-T_6(\xi_2),\ -T_6(\xi_1)=\xi_2.
\]

Hence we obtain only one equation 
\[
\xi=-T_6(-T_6(\xi))=-T_6(T_6(\xi)).
\]
This equation $\xi+T_6(T_6(\xi))=0$ is a polynomial equation of degree $36$ with distinct 36 roots as follows: 
\[-2=2\cos \pi,\ 2\cos\frac{k \pi}{35}\ (k=1,3,\dots,33),\ 2\cos\frac{k\pi}{37}\ (k=1,3,\dots,35).
\]

It is seen that they are the roots as 
\[
\begin{split}
T_6(T_6(-2))
&=T_6(T_6(2\cos\pi))\\
&=2\cos 36\pi=2,\\
T_6\left(T_6\left(2\cos\frac{k\pi}{35}\right)\right)
&=2\cos \frac{36 k\pi}{35} = -2\cos\frac{k\pi}{35},\\
T_6\left(T_6\left(2\cos\frac{k\pi}{37}\right)\right)
&=2\cos \frac{36 k\pi}{37} = -2\cos\frac{k\pi}{37}.
\end{split}
\]
Further one easily sees that $\xi=-2$ does not give a representation on the splice. 
The roots $2\cos\frac{k \pi}{35}(k=1,3,\dots,33)$ are corresponding to the condition $s_1^{36}=s_1$ 
coming from matrix equations $L_1=X_2$ and $L_2=X_1$.

It can be seen that there exists a $k$ such that 
$\tr(\rho(x_1))=2\cos\frac{k \pi}{35}$ 
and $\tr(\rho(x_2))=-T_6(2\cos\frac{k \pi}{35})$ 
for any $[\rho]\in X^\irr(\Sigma(K_1,K_2))$.
On the other hand, the roots $2\cos\frac{k \pi}{37}(k=1,3,\dots,35)$ are corresponding to the condition 
$s_1^{36}=s_1^{-1}$ 
coming from equations $L_1=X_2^{-1}$ and $L_2=X_1$. 
They give representations of the splicing of $3_1$ and its mirror image, not $3_1$. 

Take $[\rho]=[\rho_1\ast \rho_2]\in X^\irr(\Sigma(K_1,K_2))$ 
and identify it with $(X_1,Y_1,X_2,Y_2)$. 
Consider 
\[
A_a=\begin{pmatrix}
a & \frac{a-1/a}{s_1-1/s_1} \\
 0 & 1/a \\
\end{pmatrix},
\]
where $a \in\C^\times$, 
$s_1,s_2 \in\C^\times$ are satisfying 
$s_1+1/s_1=\xi_1, s_2+1/s_2=\xi_2$ and $\xi_1=-T_6(T_6(\xi_1)), \ \xi_2=-T_6(\xi_1)$.
In this case, one gets
\[
\begin{split}
& \tr\left((A_a\rho_{1}A_a^{-1})\ast (\rho_{2})(y_1 x_2 )\right) \\
=& \tr(X_2 A_a Y_1 A_a^{-1}) \\
=& s_1 s_2+\frac{1}{s_1 s_2}-\frac{c^2}{a^2}(s_1^2+1/s_1^2-1)+\frac{(1-a^2)(s_2-1/s_2)}{(s_1-1/s_1)}(s_1^2+1/s_1^2-1).\\
\end{split}
\]
Here $c$ is determined by 
$X_2=L_1$, 
namely 
\[
c^2=-\frac{(1+s_1^2+s_1^4)(1+s_1^4)}{s_1^5}.
\]

\subsection{$q=-1$ Case}

We put $q=-1$. 
This $J(2,-2)$ is the figure-eight knot. 
In this case the Riley polynomial $\phi_{-1}(s,t)$ is given by 
\[
\begin{split}
\phi_{-1}(s,t)
&= t^2-(s^2+1/s^2-3)t-s^2-1/s^2+3 \\
&= t^2-(\xi^2-5)t-\xi^2+5,
\end{split}
\]
where $\xi=s+1/s$. 

Then the irreducible representation part of $X^\irr(E(J(2,4)))$ is 
\[
\{(\xi,t)\in\C^2 \mid t^2-(\xi^2-5)t-\xi^2+5=0,\ t\neq 0\}.
\]
Under the same notations, 
one obtains 
\[
\xi_1=\xi_2^4-5\xi_2^2+2,~ \xi_1^4-5\xi_1^2+2=\xi_2.
\]
Hence we obtain only one equation 
\[
\xi=\xi^{16}-20 \xi^{14}+158\xi^{12}-620\xi^{10}+1244\xi^8-1190 \xi^6+487 \xi^4-60 \xi^2-2 
\]
and 16 roots $\nu_0=-2,\nu_i\neq \pm 1~(i=1,\dots,15)$. 
For any $\nu_i\ (i\neq 0)$, we can take a  bending construction 
to do deformations in $X^\irr(\Sigma(J(2,-2),J(2,-2)))$. 
\begin{remark}
In this case, $t$ is a root of $t^2-(\nu_i^2-5)t-\nu_i^2+5=0$.
\end{remark}



\end{document}